\documentclass[a4paper,12pt,reqno]{amsart}

\usepackage[utf8]{inputenc}
\usepackage{latexsym}
\usepackage{amssymb}
\usepackage{amsmath}
\usepackage{enumerate}
\usepackage{xcolor}
\usepackage{url}
\usepackage{mathrsfs}

\newtheorem{theorem}{Theorem}[section]
\newtheorem{lemma}[theorem]{Lemma}
\newtheorem{proposition}[theorem]{Proposition}
\newtheorem{corollary}[theorem]{Corollary}
\newtheorem{example}[theorem]{Example}

\newtheorem{problem}[theorem]{Problem}
\newcommand\aone{A_{1+}^{-1}}
\newcommand\atwo{A_{2+}^{-1}}

\newcommand{\aplumi}{A_+^{-1}}

\newcommand{\semino}{A_{+}}
\newcommand{\semione}{A_{1+}}
\newcommand{\semitwo}{A_{2+}}

\newcommand{\effone}{E(A_1)}
\newcommand{\efftwo}{E(A_2)}

\title[]{
Non-linear characterization of Jordan $*$-isomorphisms via 
maps on positive cones of $C^*$-algebras
}
\author{
Osamu~Hatori
}
\address{
Institute of Science and Technology,
Niigata University, Niigata 950-2181, Japan.
}
\email{hatori@math.sc.niigata-u.ac.jp
}

\author{
Shiho Oi
}
\address{
Department of Mathematics, Faculty of Science, 
Niigata University, Niigata 950-2181, Japan.
}
\email{shiho-oi@math.sc.niigata-u.ac.jp
}


\begin{document}

\maketitle\textbf{}

\begin{abstract}
We study maps between positive definite or positive semidefinite cones of unital $C^*$-algebras. We describe 
surjective maps that preserve
\begin{itemize}
    \item[(1)] the norm of the quotient or multiplication of elements;
    \item[(2)] the spectrum of the quotient or multiplication of elements;
    \item[(3)] the spectral seminorm of the quotient or multiplication of elements.
\end{itemize}
These maps relate to the Jordan $*$-isomorphisms between the specified $C^*$-algebras. While a surjection between positive definite cones that preserves the norm of the quotient of elements may not be extended to a linear map between the underlying $C^*$-algebras, the other types of surjections can be extended to a Jordan $*$-isomorphism or a Jordan $*$-isomorphism followed by the implementation by a positive invertible element. We also study conditions for the centrality of positive invertible elements. 
We generalize "the corollary" regarding surjections between positive semidefinite cones of unital $C^*$-algebras. Applying it, 
we provide positive solutions to the problem posed by Moln\'ar %\cite[p.194]{moljot} 
for general unital $C^*$-algebras.
\end{abstract}

\maketitle

%%%%%%%%%%%%%%%%%%%%%%%%%%%%%%%%%%%%%%%%%%%%%
%%%%%%%%%%%%%%%%%%%%%%%%%%%%%%%%%%%%%%%%%%%%%
%%%%%%%%%%%%%%%%%%%%%%%%%%%%%%%%%%%%%%%%%%%%%
\section{Introduction}
%%%%%%%%%%%%%%%%%%%%%%%%%%%%%%%%%%%%%%%%%%%%%
%%%%%%%%%%%%%%%%%%%%%%%%%%%%%%%%%%%%%%%%%%%%%
%%%%%%%%%%%%%%%%%%%%%%%%%%%%%%%%%%%%%%%%%%%%%
A Jordan $*$-isomorphism 
plays a crucial role in the study of mappings 
between $C^*$-algebras. It preserves various quantities, sets, and structures of $C^*$-algebras. Conversely, we can characterize a Jordan $*$-isomorphism by its preservation of specific quantities, sets, and structures. Kadison's theorem \cite{kad} is a celebrated result that states that a surjective complex linear unital isometry between unital $C^*$-algebras is a Jordan $*$-isomorphism. Kadison also showed that a unital order isomorphism is a Jordan $*$-isomorphism \cite{kad2}. Therefore, if a unital complex linear bijection preserves positive elements, then it is a Jordan $*$-isomorphism. If a surjective complex linear map preserves the unitary group, it is a Jordan $*$-isomorphism multiplied by a unitary \cite{rd}. 

On the other hand, Moln\'ar \cite{molmul} initiated the study of the so-called multiplicatively spectrum-preserving maps. A map $T\colon B_1\to B_2$ between Banach algebras $B_1$ and $B_2$ is called multiplicatively spectrum-preserving if $\sigma(T(a)T(b))=\sigma(ab)$ for $a,b\in B_1$. He presented non-linear characterizations of automorphisms on algebras of functions and operators. Several further investigations were obtained in this way. 
We mention only a few of the corresponding works
 \cite{rr2,rr,hmt,hmt2,lt,tl,hs,hllmty} on commutative Banach algebras and non-commutative Banach algebras \cite{cls,hlw,mho,bms,aa,moljot,ht,dang,tbs}.

In this paper, we study non-linear characterizations of Jordan $*$-isomorphisms on positive (definite or semidefinite) cones of $C^*$-algebras. Here, "non-linear" means that we do not assume that the maps we are considering are linear in any sense. 
We investigate surjective maps that preserve the norm, the spectral seminorm, and the spectrum of the multiplication as well as the quotient of elements between positive definite cones of $C^*$-algebras. 
The positive (definite or semidefinite) cone has a vibrant structure from algebraic and geometrical points of view. It has wide-ranging applications in various areas of mathematics and mathematical physics. 

Moln\'ar's fascinating lectures \cite{mollec} inspired our research on preserver problems related to algebras of operators and functions. During these lectures, he presented his and his colleague's findings about preservers of the norm of the arithmetic mean and the geometric mean of positive invertible elements in a unital $C^*$-algebra (cf. \cite{mol,cmm,dongli}). He presented non-linear characterizations of Jordan $*$-isomorphisms. He also discussed these maps during the lectures. All these findings inspired us to study surjective maps between positive (definite or semidefinite) cones of $C^*$-algebras.

We will start by clarifying the notation and introducing the necessary definitions and properties that we will use throughout the paper. 
In this paper, we use $A$, $A_1$, and $A_2$ to denote unital $C^*$-algebras. We always write the unit by $e$. 
%%%%%%%%%%%%%%%%%%%%%%
%%%%%%%%%%%%%%%%%%%%%%
%%%%%%%%%%%%%%%%%%%%%%
We denote 
\[
A_{SA}=\{a\in A\colon a=a^*\},
%\text{$a$ is self-adjoint}\},
\]
the Jordan algebra of all self-adjoint elements in $A$. 
The positive semidefinite cone is
\[
\semino=\{a\in A_{SA}\colon a\ge 0\},
\]
and the positive definite cone is
\[
\aplumi=\{a\in \semino\colon \text{$a$ is invertible in $A$}\}.
\]
%%%%%%%%%%%%%%%%%%%%%%
%%%%%%%%%%%%%%%%%%%%%%
%%%%%%%%%%%%%%%%%%%%%%
The spectrum of $a\in A$ is denoted by $\sigma(a)$, and the spectral seminorm on $A$ is 
\[
\|a\|_S=\sup\{|t|\colon t\in \sigma(a)\}.
\]
%For any $a\in A$, we use $\sigma(\cdot)$ to represent the spectrum of $a$ in $A$. The spectral seminorm on $A$ is denoted by $\|\cdot\|_S$, which is defined as $\|x\|_S=\sup\{|t|\colon t\in \sigma(x)\}$ for $x\in A$. 
It is worth noting that $\|a\|=\|a\|_S$ for a self-adjoint element $a$, in particular, $a\in A_+^{-1}$.
A Jordan $*$-isomorphism is a complex linear bijective map between two unital $C^*$-algebras that preserves the Jordan product ($(ab+ba)/2$ for $a,b$) and the involution. 
 It preserves squaring,  
 the unit, and the invertibility. It also preserves commutativity. The same applies for the spectrum. Although a Jordan $*$-isomorphism may not preserve multiplication, it is known to preserve the triple product: that is, $J(aba)=J(a)J(b)J(a)$, for $a,b$ in the domain.   
It is worth noting that a Jordan $*$-isomorphism preserves the norm. Kadison's theorem is a celebrated result stating that if a complex linear unital bijection preserves the norm, it is a Jordan $*$-isomorphism \cite[Theorem 7]{kad}. 
A complex linear unital bijection is a Jordan $*$-isomorphism if it preserves the order in both directions, as stated in \cite[Corollary 5]{kad2}.

%Tour\'e, Brits and Sebastian \cite[Corollary 3.9]{tbs}

Moln\'ar presented the following striking result in \cite{mol}. Let $\phi\colon \aone\to \atwo$ be a map. We say that $\phi$ is positive homogeneous if the condition $\phi(ta)=t\phi(a)$ holds for all $a\in \aone$ and real numbers $t>0$. 
We say that $\phi$ is order-preserving in both directions provided that $a\le b$ if and only if $\phi(a)\le \phi(b)$ for any $a,b\in \aone$.
We say that $\phi$ is an order isomorphism if  $\phi$ is a bijection and order-preserving in both directions.   
It is worth noting that if $\phi$ is a surjection and order-preserving in both directions, then $\phi$ is also an injection. 
(In fact, if $\phi(x)=\phi(x')$, then $\phi(x)\le \phi(x')$ and $\phi(x')\le \phi(x)$. Hence
we have $x\le x'$ and $x'\le x$ simultaneously. Thus, $x=x'$ means $\phi$ is an injection.) Thus, $\phi$ is a bijection, hence an order isomorphism. 

%%%%%%%%%%%%%%%%%%%%%%%%%%%%%%%%%%%%%%%%%%%%%
\begin{theorem}[Proposition 13 in \cite{mol}]\label{thecor}
Suppose that $\phi\colon \aone\to \atwo$ is a positive homogeneous order isomorphism. Then there exists a Jordan $*$-isomorphism $J\colon A_1\to A_2$ such that $\phi=\phi(e)^\frac12J\phi(e)^\frac12$ on $\aone$.
\end{theorem}
%%%%%%%%%%%%%%%%%%%%%%%%%%%%%%%%%%%%%%%%%%%%%
After examining the definition of the Thompson metric, we can conclude that if $\phi$ is a positive homogeneous order isomorphism, it is also a Thompson isometry. By utilizing \cite[Theorem 9]{hm} and carefully considering the matter at hand, we can prove Theorem \ref{thecor}. 
Theorem \ref{thecor} is highly applicable in many situations, including in this paper. Furthermore, it should be mentioned that \cite[Theorem 9]{hm} is proven by utilizing a profound result of Kadison \cite[Theorem 7]{kad} mentioned above. It is worth pointing out that there is a typo in \cite[Theorem 13]{mol}, where $\phi\colon {\mathscr{A}}\to {\mathscr{B}}$ should read $\phi\colon {\mathscr{A}}_+^{-1}\to {\mathscr{B}}_+^{-1}$.
Moln\'ar refers to Theorem \ref{thecor} "the corollary" in his lectures \cite{mollec} because several interesting results follow from it. We also repeatedly apply this valid "the corollary" in this paper.

In section \ref{s2}, we present characterizations of bijections between positive definite cones that preserve the norm of the quotients of positive elements (Theorem \ref{qim}). This bijection is related to a Jordan $*$-isomorphism as usual. However, it may not be extended to a linear map. We also provide a condition pertaining to centrality that allows the underlying bijection to be extended to a Jordan $*$-isomorphism. Theorem \ref{hoo} is about a bijection that preserves the spectrum or spectral seminorm of the quotient of positive invertible elements. This bijection is extended to a Jordan $*$-isomorphism followed by the implementation of a positive invertible element. 

In section  \ref{s3}, we are concerned with multiplicatively norm or spectrum-preserving bijections. 
As described in the second paragraph of this section, Moln\'ar \cite{molmul} initiated the study of multiplicatively spectrum-preserving maps between algebras of operators or functions. We study multiplicatively norm-preserving maps between positive definite cones of $C^*$-algebras (Theorem \ref{3.6}). 
We are also concerned with maps that preserve the norm of the triple product.

In section \ref{central}, we investigate the conditions of centrality for positive invertible elements. Ogasawara \cite{oga} proved that a $C^*$-algebra is commutative if and only if squaring preserves the order for all positive semidefinite elements. Moln\'ar introduced local monotonicity \cite{molmono}. 
A function $f$ is said to be locally monotone at a self-adjoint element $a$ if $a\le x$ implies $f(a)\le f(x)$ for every self-adjoint element $x$. 
He proved that a self-adjoint element $a$ is central if and only if the exponential function is locally monotone at $a$. 
Finally, Nagy \cite{nagy} succeeded in proving that every strictly convex increasing function defined on an open interval unbounded from above is locally monotone at $a$ if and only if it is central. Prior to that, Virosztek \cite{virosz} had proved the corresponding result for a large class of functions. 
The proofs by Nagy and Virosztek are universally applicable to a broad category of functions, encompassing the squaring function. However, it requires a rather extensive computation. It is noteworthy to present a proof specifically for the squaring function only if it is sufficiently straightforward.
We present a concise and straightforward direct proof establishing that the squaring function is locally monotone at a positive invertible element if and only if it is central (Proposition \ref{lo}). 
Through its application, we illustrate specific norm conditions that are sufficient for the centrality of positive invertible elements.

In section 5, we study maps between positive semidefinite cones. We prove "the corollary" for maps between semidefinite cones (Theorem \ref{exthe}). Applying Theorem \ref{exthe}, we exhibit similar results of section 4 for the case of positive semidefinite cones. In particular, we show a generalization of Theorem 2.6 and Corollary 2.9 in \cite{moljot} for the case of general unital $C^*$-algebras. Our results provide positive solutions to the problem posed by Moln\'ar \cite[p.194]{moljot}.

%The final section is devoted to a related problem throughout the paper. 
%%%%%%%%%%%%%%%%%%%%%%%%%%%%%%%%%%%%%%%%%%%%%
%%%%%%%%%%%%%%%%%%%%%%%%%%%%%%%%%%%%%%%%%%%%%
%%%%%%%%%%%%%%%%%%%%%%%%%%%%%%%%%%%%%%%%%%%%%
\section{Maps that preserve the spectrum or norm of the quotients of elements in positive definite cones}\label{s2}
%%%%%%%%%%%%%%%%%%%%%%%%%%%%%%%%%%%%%%%%%%%%%
%%%%%%%%%%%%%%%%%%%%%%%%%%%%%%%%%%%%%%%%%%%%%
%%%%%%%%%%%%%%%%%%%%%%%%%%%%%%%%%%%%%%%%%%%%%
We aim to prove that the map $x\mapsto (ax^2a)^\frac12a^{-1}$ on $A_+^{-1}$ preserves the norm in Proposition \ref{gyoe}, for any $a\in A_+^{-1}$. We can extend this map to a linear map on $A$ if and only if $a$ is central in $A$ (Proposition \ref{ax2a}). Using this approach, we prove the main result of this section, Theorem \ref{qim}. We observe that a bijection between positive definite cones, which preserves the norm of the quotient of any two elements, is related to a Jordan $*$-isomorphism between the whole algebras. Nevertheless, this bijection generally cannot be extended to a linear map between entire algebras.

In the first draft of the paper, we apply the Gelfand-Naimark theorem to prove the following proposition. However, Lajos Moln\'ar pointed out that direct proof is possible. This adjustment not only streamlines the proof but also enhances the clarity and accessibility of our paper.
%%%%%%%%%%%%%%%%%%%%%%%%%%%%%%%%%%%%%%%%%%%%
\begin{proposition}\label{gyoe}
    For any trio $a,x, y\in A_+^{-1}$ we have
    \[
    \|(ax^2a)^\frac12a^{-1}\|=\|x\|
    \]
    and
    \[
    \|(ax^2a)^\frac12(ay^2a)^{-\frac12}\|=\|x y^{-1}\|.
    \]
\end{proposition}
%%%%%%%%%%%%%%%%%%%%%%%%%%%%%%%%%%%%%%%%%%%%
\begin{proof}
We have 
\begin{multline*}
\|(ax^2a)^\frac12a^{-1}\|^2=\|((ax^2a)^\frac12a^{-1})^*((ax^2a)^\frac12a^{-1})\|
\\
=\|a^{-1}(ax^2a)a^{-1}\|=\|x^2\|=\|x\|^2.
\end{multline*}
Thus the first equation holds. For the second equation, we have
\begin{multline*}
    \|(ax^2a)^\frac12(ay^2a)^{-\frac12}\|^2
=
\|((ax^2a)^\frac12(ay^2a)^{-\frac12})^*((ax^2a)^\frac12(ay^2a)^{-\frac12})\|
\\=
\|(ay^2a)^{-\frac12}ax^2a(ay^2a)^{-\frac12}\|
=
\|xa(ay^2a)^{-\frac12}\|^2
\\
=
\|(xa(ay^2a)^{-\frac12})(xa(ay^2a)^{-\frac12})^*\|
\\
=
\|xa(ay^2a)^{-1}ax\|=\|xy^{-2}x\|=\|xy^{-1}\|^2.
\end{multline*}
\end{proof}
%%%%%%%%%%%%%%%%%%%%%%%%%%%%%%%%%%%%%%%%%%%%%%%%%%%%%
To prove the main result of this section, Theorem \ref{qim}, we will utilize the following proposition.
%%%%%%%%%%%%%%%%%%%%%%%%%%%%%%%%%%%%%%%%%%%%%
\begin{proposition}\label{ax2a}
Let $a\in A_+^{-1}$. 
    Suppose that $\phi\colon A_+^{-1}\to A_+^{-1}$ is defined as $\phi(x)=(ax^2a)^\frac12$, $x\in A_+^{-1}$. 
    Then $\phi$ is additive if and only if $a$ is a central element in $A$. In this case, 
     $\phi(x)=ax$, $x\in A_+^{-1}$. 
\end{proposition}
%%%%%%%%%%%%%%%%%%%%%%%%%%%%%%%%%%%%%%%%%%%%%%
\begin{proof}
If $a$ is a central element, then by an elementary calculation, we have $\phi(x)=ax$ for every $x\in A_+^{-1}$, so $\phi$ is additive. 

    Suppose that $\phi(x)=(ax^2a)^\frac12$, $x\in  A_+^{-1}$ is an additive map. Then 
    \[
    (a(x^2+xy+yx+y^2)a)^{\frac12}=(a(x+y)^2a)^\frac12=(ax^2a)^\frac12+(ay^2a)^\frac12
    \]
    for every $x,y\in A_+^{-1}$. To simplify the equation, we can square both sides, resulting in: 
\[
axya+ayxa=(ax^2a)^{\frac{1}{2}}(ay^2a)^{\frac{1}{2}}+(ay^2a)^{\frac{1}{2}}(ax^2a)^{\frac{1}{2}}.
\]
    Substituting $x=a^{-1}$ we obtain 
    \begin{equation}\label{444}
    ya+ay=2(ay^2a)^\frac12, \quad y\in A_+^{-1}.
    \end{equation}
    Since $\phi$ is additive, the inverse map $\phi^{-1}(x)=(a^{-1}x^2a^{-1})^\frac12$, $x\in A_+^{-1}$ is also additve. In a similar way as \eqref{444} we obtain
    \begin{equation}\label{555}
    ya^{-1}+a^{-1}y=2(a^{-1}y^2a^{-1})^\frac12, \quad y\in A_+^{-1}.
    \end{equation}
    Replacing $y$ by $aya$ in \eqref{555}, we have
    \begin{equation}\label{666}
     ay+ya=2(ya^2y)^\frac12, \quad y\in A_+^{-1}.
    \end{equation}
    By compairing \eqref{444} and \eqref{666} we have
    \[
    (ay)(ay)^*=ay^2a=ya^2y=(ay)^*(ay), \quad y\in A_+^{-1}. 
    \]
    Thus $ay$ is normal for every $y\in A_+^{-1}$. On the other hand, the equation 
    \[
    \sigma(ay)=\sigma(y^\frac12ay^\frac12)\subset (0,\infty)
    \]
    implies that the spectrum of $ay$ consists of positive real numbers. 
    It may be well known that a normal element, whose spectrum is confined to real numbers, is self-adjoint. 
This can be demonstrated by examining the closed subalgebra generated by the corresponding normal element, constituting a commutative $C^*$-algebra. As the spectra in a $C^*$-algebra and in a closed $*$-subalgebra are identical, and in a commutative $C^*$-algebra, an element whose spectrum is confined to the real numbers is self-adjoint, it follows that the element $ay$ is self-adjoint.     
     This implies that $a$ and $y$ commute for every $y \in A_+^{-1}$, which further implies that $a$ is a central element in $A$. Finally, a simple calculation shows that $\phi(x)=ax$ for $x \in A_+^{-1}$.
\end{proof}
%%%%%%%%%%%%%%%%%%%%%%%%%%%%%%%%%%%%%%%%%%%%%%%%%
\begin{theorem}\label{qim}
    Let $\phi\colon \aone\to \atwo$ be a surjection. Then, the following are equivalent:
    \begin{itemize}
        \item[(i)] $\|\phi(x)\phi(y)^{-1}\|=\|xy^{-1}\|$, \quad $x,y\in \aone$,
        \item[(ii)] there exists a Jordan $*$-isomorphism $J\colon A_1\to A_2$ such that 
        \[
        \phi(x)=(\phi(e)J(x)^2\phi(e))^\frac12, \quad x\in \aone.
        \]
    \end{itemize}
    When $\phi$ satisfies the condition {\rm(i)}, or equivalently {\rm(ii)}, the map $\phi$ is additive if and only if $\phi(e)$ is a central element in $A_2$. In this case, 
    \[
    \phi(x)=\phi(e)J(x), \quad x\in \aone.
    \]
    \end{theorem}
%%%%%%%%%%%%%%%%%%%%%%%%%%%%%%%%%%%%%%%%%%%%%%%%%%%%%%%%%
\begin{proof}
We prove that (i) implies (ii). 
    Define $\psi\colon \aone\to \atwo$ by $\psi(x)=\phi(x^\frac12)^2$, $x\in \aone$. Then $\psi$ is a surjection. By calculation, we have
    \begin{multline*}
       \|y^{-\frac12}x^\frac12\|=\|(x^\frac12y^{-\frac12})^*\|=\|x^\frac12y^{-\frac12}\|
       =
       \|\phi(x^\frac12)\phi(y^{\frac12})^{-1}\|
      \\
      =
       \|\psi(x)^\frac12\psi(y)^{-\frac12}\|
       =
       \|\psi(y)^{-\frac12}\psi(x)^\frac12\|.
    \end{multline*}
        Then we have
        \begin{multline}\label{a1}
            \|y^{-\frac12}xy^{-\frac12}\|=
            \|(y^{-\frac12}x^\frac12)(y^{-\frac12}x^\frac12)^*\|=\|y^{-\frac12}x^\frac12\|^2
            \\
            =
            \|\psi(y)^{-\frac12}\psi(x)^\frac12\|^2
            =
            \|\psi(y)^{-\frac12}\psi(x)\psi(y)^{-\frac12}\|.
        \end{multline}
        Similarly, we have
        \begin{equation}\label{a2}
        \|x^{-\frac12}yx^{-\frac12}\|=\|\psi(x)^{-\frac12}\psi(y)\psi(x)^{-\frac12}\|.
        \end{equation}
    Since $y^{-\frac12}xy^{-\frac12}$ is positive, we have
    \begin{multline}\label{a3}
        \|y^{-\frac12}xy^{-\frac12}\|=\sup\{t\colon t\in \sigma(y^{-\frac12}xy^{-\frac12})\}
        \\
        =
        \inf\{t\colon y^{-\frac12}xy^{-\frac12}\le te\}
        =
        \inf\{t\colon x\le ty\}.
    \end{multline}
    In the same way, we have
    \begin{equation}\label{a4}
    \|x^{-\frac12}yx^{-\frac12}\|=\inf\{s\colon y\le sx\}.
    \end{equation}
    Recall that the Thompson metric $d_T(\cdot,\cdot)$ on $A_+^{-1}$ is   
    \[
        d_T(x,y)=\log\max\{\inf\{t\colon x\le ty\}, \inf\{s\colon y\le sx\}\}, \quad x,y\in A_+^{-1}.
    \]
    It follows by \eqref{a1}, \eqref{a2}, \eqref{a3} and \eqref{a4} that for every $x,y\in \aone$
    \begin{multline*}
        d_T(x,y)=\log\max\{\|x^{-\frac12}yx^{-\frac12}\|, \|y^{-\frac12}xy^{-\frac12}\|\}
        \\
        =
        \log\max\{\|\psi(x)^{-\frac12}\psi(y)\psi(x)^{-\frac12}\|, \|\psi(y)^{-\frac12}\psi(x)\psi(y)^{-\frac12}\|\}
        \\
        =d_T(\psi(x),\psi(y)).
    \end{multline*}
    We conclude that $\psi\colon \aone\to \atwo$ is a surjective Thompson isometry. By \cite[Theorem 9]{hm} there exists a Jordan $*$-isomorphism $J\colon A_1\to A_2$ and a central projection $P\in A_2$ such that 
    \begin{equation}\label{111}
    \psi(x)=\psi(e)^\frac12(PJ(x)+(e-P)J(x^{-1}))\psi(e)^\frac12, \quad x\in \aone.
    \end{equation}
    Substituting $x=\frac12 e$ and $y=e$ in \eqref{a1}, we obtain by \eqref{111} that 
    \begin{multline*}
        \frac12=\left\|e^{-\frac12}\frac12ee^{-\frac12}\right\|
        =
        \left\|\psi(e)^{-\frac12}\psi\left(\frac12e\right)\psi(e)^{-\frac12}\right\|
        \\
        =
        \left\|PJ\left(\frac12e\right)+(e-P)J\left(\left(\frac12e\right)^{-1}\right)\right\|
        =
        \left\|\frac12P+2(e-P)\right\|,
    \end{multline*}
    as $P$ is a central projection, we have
    \[
    =\max\{\|\frac12P\|,\|2(e-P)\|\}.
    \]
    It follows that $e-P=0$, so $\psi=\psi(e)^\frac12J\psi(e)^\frac12$. We conclude that 
    \[
    \phi(x)=\psi(x^2)^\frac12=(\phi(e)J(x)^2\phi(e))^\frac12, \quad x\in \aone.
    \]

    Suppose that (ii) holds. We prove (i). Applying Proposition \ref{gyoe} we have
    \begin{multline}\label{333}
    \|\phi(x)\phi(y)^{-1}\|=\|(\phi(e)J(x)^2\phi(e))^\frac12(\phi(e)J(y)^2\phi(e))^{-\frac12}\|
    \\
    =\|J(x)J(y)^{-1}\|, \quad x,y\in \aone.
    \end{multline}
    On the other hand, as a Jordan $*$-isomorphism preserves squaring, the triple product, the inverse, and the norm, we have
    \begin{multline*}
        \|J(x)J(y)^{-1}\|^2=\|(J(x)J(y)^{-1})^*J(x)J(y)^{-1}\|
        \\
        =\|J(y)^{-1}J(x^2)J(y)^{-1}\|=\|J(y^{-1}x^2y^{-1})\|
        \\
        =\|y^{-1}x^2y^{-1}\|=\|xy^{-1}\|^2, \quad x,y\in \aone.
    \end{multline*}
    Then by \eqref{333} we conclude 
    \[
    \|\phi(x)\phi(y)^{-1}\|=\|xy^{-1}\|
    \]
    holds for every $x,y\in \aone$.

    Suppose that $\phi(x)=(\phi(e)J(x)^2\phi(e))^\frac12$, $x\in \aone$ and $\phi$ is additive. This implies that the mapping $y \mapsto (\phi(e)y^2\phi(e))^\frac12$ is additive. Therefore, by Proposition \ref{ax2a}, we conclude that $\phi(e)$ is a central element in $A_2$. We infer by a simple calculation that $\phi(x)=\phi(e)J(x)$ for every $x\in \aone$.
\end{proof}
%%%%%%%%%%%%%%%%%%%%%%%%%%%%%%%%%%%%%%%%%%%%%%%%%%%%%%%%%%
Note that if $\|\phi(x)\phi(y)^{-1}\|=\|xy^{-1}\|$ holds for every $x,y\in \aone$ and $\phi(e)$ is not a central element in $A_2$, then $\phi$ is not extended to even an additive map from $A_1$ into $A_2$.

    The structure of a surjection that preserves the spectrum or the spectrum seminorm of the quotient is relatively simple. It can be extended to a Jordan $*$-isomorphism followed by the implementation of an element in $\atwo$.
    To elaborate, the following holds true.
%%%%%%%%%%%%%%%%%%%%%%%%%%%%%%%%%%%%%%%%%%%%%%%%%%%%%%%%%%%%%
\begin{theorem}\label{hoo}
Let $\phi\colon\aone\to \atwo$ be a surjection. Then, the following are equivalent:
\begin{itemize}
    \item[(i)]
    There exists a Jordan $*$-isomorphism $J\colon A_1\to A_2$ such that 
    \[
    \phi(x)=\phi(e)^\frac12J(x)\phi(e)^\frac12, \quad x\in \aone,
    \]
        \item[(ii)]
    $\|\phi(x)\phi(y)^{-1}\|_S=\|xy^{-1}\|_S, \quad x,y\in \aone$,
    \item[(iii)]
    $\sigma(\phi(x)\phi(y)^{-1})=\sigma(xy^{-1}),\quad x,y\in \aone$. 
\end{itemize}
\end{theorem}
%%%%%%%%%%%%%%%%%%%%%%%%%%%%%%%%%%%%%%%%%%%%%%%%%%%%%%%%%%%%
\begin{proof}
    By the definition of the spectrum norm, it is evident that (iii) implies (ii). 

    We prove that (ii) implies (i). For every $x,y\in \aone$, 
    \begin{multline*}
        \|xy^{-1}\|_S=\sup\{t\colon t\in \sigma(xy^{-1})\}
        =
        \sup\{t\colon t\in \sigma(y^{-\frac12}xy^{-\frac12})\}
        \\
        =
        \inf\{t\colon y^{-\frac12}xy^{-\frac12}\le te\}
        =
        \inf\{t\colon x\le ty\}.
    \end{multline*}
    In the same way, $\|yx^{-1}\|_S=\inf\{t\colon y\le tx\}$, 
    $\|\phi(x)\phi(y)^{-1}\|_S=\inf\{t\colon \phi(x)\le t\phi(y)\}$, and 
    $\|\phi(y)\phi(x)^{-1}\|_S=\inf\{t\colon \phi(y)\le t\phi(x)\}$.
    Due to the definition of the Thompson metric, we obtain
    \begin{multline*} 
    d_T(x,y)=\log\max\{\|xy^{-1}\|_S, \|yx^{-1}\|_S\}
    \\
    =
    \log\max\{\|\phi(x)\phi(y)^{-1}\|_S,\|\phi(y)\phi(x)^{-1}\|_S\}
    =
    d_T(\phi(x),\phi(y))
    \end{multline*}
    for every $x,y\in \aone$; $\phi$ is a (surjective) Thompson isometry. 
    Hence, according to \cite[Theorem 9]{hm} there exist a Jordan $*$-isomorphism $J$ and a central projection $P\in A_2$ such that 
    \[
    \phi(x)=\phi(e)^\frac12(PJ(x)+(e-P)J(x^{-1}))\phi(e)^\frac12, \quad x\in \aone. 
    \]
    Letting $x=\frac12e$, we have 
    \[
    \phi\left(\frac{1}{2}e\right)=\phi(e)^\frac12\left(\frac12 P+ 2(e-P)\right)\phi(e)^\frac12.
    \]
    So
    \begin{multline*}
    \frac12=\left\|\frac12 ee^{-1}\right\|_S=\left\|\phi(e)^\frac12\left(\frac12 P+2(e-P)\right) \phi(e)^{-\frac12}\right\|_S
    \\
    =
    \left\|\frac12P+2(e-P)\right\|_S.
    \end{multline*}
    It follows that $e-P=0$. We conclude that $\phi=\phi(e)^\frac12J\phi(e)^\frac12$. 

    Suppose that (i) holds. Then for every $x,y\in \aone$, we infer $\phi(x)\phi(y)^{-1}=\phi(e)^\frac12J(x)J(y^{-1})\phi(e)^{-\frac12}$, so
    \begin{multline*}      \sigma(\phi(x)\phi(y)^{-1})=\sigma(\phi(e)^\frac12J(x)J(y^{-1})\phi(e)^{-\frac12})
      \\
      =
      \sigma(J(x)J(y^{-1}))=\sigma(J(x^\frac12)J(y^{-1})J(x^\frac12))
      \\
      =      \sigma(J(x^\frac12y^{-1}x^\frac12))=\sigma(x^\frac12y^{-1}x^\frac12)=\sigma(xy^{-1}).
    \end{multline*}
    Thus, (iii) holds.
\end{proof}
%%%%%%%%%%%%%%%%%%%%%%%%%%%%%%%%%%%%%%%%%%%%%
%%%%%%%%%%%%%%%%%%%%%%%%%%%%%%%%%%%%%%%%%%%%%
%%%%%%%%%%%%%%%%%%%%%%%%%%%%%%%%%%%%%%%%%%%%%
\section{Maps that preserve the spectrum or the norm of the multiplications of elements in positive definite  cones}\label{s3}
%%%%%%%%%%%%%%%%%%%%%%%%%%%%%%%%%%%%%%%%%%%%%
%%%%%%%%%%%%%%%%%%%%%%%%%%%%%%%%%%%%%%%%%%%%%
%%%%%%%%%%%%%%%%%%%%%%%%%%%%%%%%%%%%%%%%%%%%%
In this section, we study not only multiplicatively spectrum-preserving maps but also multiplicatively norm-preserving maps and multiplicatively spectrum seminorm-preserving maps on a positive definite cone. 
We prove the following as an application of Lemma 13 in \cite{cmm}. 
%%%%%%%%%%%%%%%%%%%%%%%%%%%%%%%%%%%%%%%%%%%%%
\begin{lemma}\label{1}
    Let $a, a'\in A_+^{-1}$. 
    Then, the following are equivalent:
    \begin{itemize}
        \item[(i)] $a=a'$,
        \item[(ii)] $\|ax\|=\|a'x\|$ for every $x\in A_+^{-1}$,
        \item[(iii)] $\|xax\|=\|xa'x\|$ for every $x\in A_+^{-1}$.
    \end{itemize}
    \end{lemma}
%%%%%%%%%%%%%%%%%%%%%%%%%%%%%%%%%%%%%%%%%%%%
\begin{proof}
If $a=a'$, then (ii) and (iii) are apparent. Suppose that (iii) holds. Then, $a\le a'$ follows from $\|xax\|\le \|xa'x\|$, $x\in A_+^{-1}$ by \cite[Lemma 13]{cmm}. In the same manner, we infer that $a'\le a$. Thus, we obtain (i). Suppose that (ii) holds. Since $(ax)^*(ax)=xa^2x$ for every $x\in A_+^{-1}$ as $x$ and $a$ are self-adjoint, so
\[
\|xa^2x\|=\|ax\|^2=\|a'x\|^2=\|xa'^2x\|, \quad x\in A_+^{-1}. 
\]
As (iii) implies (i), we have $a^2=a'^2$. Hence, $a=a'$.
    
\end{proof}
%%%%%%%%%%%%%%%%%%%%%%%%%%%%%%%%%%%%%%%%%%%%%%
\begin{lemma}\label{2}
Let $a,a'\in A_+^{-1}$. Then, $a\le a'$ if and only if $\|ay\|_S\le\|a'y\|_S$ for every $y\in A_+^{-1}$. 
\end{lemma}
%%%%%%%%%%%%%%%%%%%%%%%%%%%%%%%%%%%%%%%%%%%%%%
\begin{proof}
    Suppose that $a\le a'$. Then $y^\frac12ay^\frac12\le y^\frac12a'y^\frac12$ for every $y\in A_+^{-1}$. Since $\sigma(ay)=\sigma(y^\frac12ay^\frac12)$ (resp. $\sigma(a'y)=\sigma(y^\frac12a'y^\frac12)$) for every $y\in A_+^{-1}$, which is a subset of the set of 
 positive real numbers, we obtain
    \begin{multline*}
    \|ay\|_S=
    \sup\{t\colon t\in \sigma(ay)\}=
    \sup\{t\colon t\in \sigma(y^\frac12ay^\frac12)\}
    \\
    \le
    \sup\{t\colon t\in \sigma(y^\frac12a'y^\frac12)\}
    =\sup\{t\colon t\in \sigma(a'y)\}=\|a'y\|_S
    \end{multline*}
    for every $y\in A_+^{-1}$.

    Conversely, suppose that $\|ay\|_S\le \|a'y\|_S$ for every $y\in A_+^{-1}$. 
    Since $\sigma(ay^2)=\sigma(yay)$  (resp. $\sigma(a'y^2)=\sigma(ya'y)$) and $\|yay\|=\|yay\|_S$ (resp. $\|ya'y\|=\|ya'y\|_S$) for every $y\in A_+^{-1}$, 
we have 
\[
\|yay\|=\|ay^2\|_S \le \|a'y^2\|_S=\|ya'y\|
\]
for every $y \in   A_+^{-1}$. By \cite[Lemma 13]{cmm}, we have $a \le a'$.
    \end{proof}
%%%%%%%%%%%%%%%%%%%%%%%%%%%%%%%%%%%%%%%%%%%%%%
\begin{corollary}\label{3}
    Let $a,a'\in A_+^{-1}$. Then $a=a'$ if and only if $\|ay\|_S=\|a'y\|_S$ holds for every $y\in A_+^{-1}$.
\end{corollary}
%%%%%%%%%%%%%%%%%%%%%%%%%%%%%%%%%%%%%%%%%%%%
A proof is apparent from Lemma \ref{2} and is omitted.
%%%%%%%%%%%%%%%%%%%%%%%%%%%%%%%%%%%%%%%%%%%%%
\begin{lemma}\label{34}
Suppose that $\phi$ and $\psi$ are surjections from $\aone$ onto $\atwo$ such that 
\begin{equation*}\label{yxxy} 
\|\phi(y)\psi(x)\phi(y)\|=\|yxy\|,\quad x,y\in \aone
\end{equation*}   
and $\psi(e)=e$.
Then, there exists a Jordan $*$-isomorphism $J\colon A_1\to A_2$ such that $\psi=J$ on $\aone$.
\end{lemma}
%%%%%%%%%%%%%%%%%%%%%%%%%%%%%%%%%%%%%%%%%%%%%
\begin{proof}
    We show that $\psi$ is order preserving in both directions. By Lemma 13 in \cite{cmm} we have $x\le x'$ if and only if $\|yxy\|\le \|yx'y\|$ if and only if 
    $\|\phi(y)\psi(x)\phi(y)\|\le \|\phi(y)\psi(x')\phi(y)\|$ if and only if $\psi(x)\le \psi(x')$ as $\phi(\aone)=\atwo$ for every $x,x'\in \aone$. This means that $\psi$ is order preserving in both directions. Hence, we see that $\psi$ is an injection, so a bijection. We prove that $\psi$ is positive homogeneous. Let $t>0$ and $x\in \aone$ be arbitrary. We have
    \begin{multline*}    \|\phi(y)\psi(tx)\phi(y)\|=\|ytxy\|
    \\
    =t\|yxy\|=t\|\phi(y)\psi(x)\phi(y)\|=\|\phi(y)t\psi(x)\phi(y)\|
    \end{multline*}
    for every $y\in \aone$. Since $\phi(\aone)=\atwo$, we have $\psi(tx)=t\psi(x)$ by Lemma \ref{1}. 
    As $t$ and $x$ are arbitrary, we see that $\psi$ is positive homogeneous. By Theorem \ref{thecor} there exists a Jordan $*$-isomorphism $J\colon A_1\to A_2$ such that $\psi=J$ on $\aone$ since $\psi(e)=e$.
\end{proof}
%%%%%%%%%%%%%%%%%%%%%%%%%%%%%%%%%%%%%%%
As a result of Lemma \ref{34}, we characterize a surjective map that preserves the norm of the triple product.
%%%%%%%%%%%%%%%%%%%%%%%%%%%%%%%%%%%%%%%%%%%%%%%%%%%%%%%%%%%%%%%%%%%%%%%
\begin{corollary}
    Let $\phi\colon \aone\to \atwo$ be a surjection. Suppose that 
    \begin{equation}\label{yyy}
    \|\phi(y)\phi(x)\phi(y)\|=\|yxy\|,\quad x,y\in \aone. 
    \end{equation}
    Then there exists a Jordan $*$-isomorphism $J\colon A_1\to A_2$ such that $\phi=J$ on $\aone$. Conversely, a Jordan $*$-isomorphism $J$ satisfies
    \[
    \|J(y)J(x)J(y)\|=\|yxy\|,\quad x,y\in \aone.
    \]
\end{corollary}
%%%%%%%%%%%%%%%%%%%%%%%%%%%%%%%%%%%%%%%%%%%%%%%%%%%%%%%%%
\begin{proof}
    Letting $x=y$ in \eqref{yyy}, we have $\|\phi(y)^3\|=\|y^3\|$. 
    As $y$ and $\phi(y)$ are positive, we have $\|\phi(y)\|=\|y\|$. 
    Hence $\|\phi(y)^2\|=\|y^2\|$. 
    Thus 
    \[
    \|\phi(y)\phi(e)\phi(y)\|=\|yey\|=\|\phi(y)^2\|=\|\phi(y)e\phi(y)\|
    \]
    for every $y\in \aone$. As $\phi(\aone)=\atwo$, applying Lemma \ref{1} we have $\phi(e)=e$. Then by Lemma \ref{34} we observe that there exists a Jordan $*$-isomorphism $J\colon A_1\to A_2$ such that $\phi=J$ on $\aone$. 

    Conversely, as a Jordan $*$-isomorphism preserves the triple product and the norm, we see that 
    \[
    \|J(y)J(x)J(y)\|=\|J(yxy)\|=\|yxy\|
    \]
    for every $x,y\in \aone$.
\end{proof}

%%%%%%%%%%%%%%%%%%%%%%%%%%%%%%%%%%%%%%%%%%
\begin{theorem}\label{3.6}
    Suppose that $\phi\colon \aone\to \atwo$ is a surjection. The following are equivalent:
    \begin{itemize}
        \item[(i)] There exists a Jordan $*$-isomorphism $J\colon A_1\to A_2$ such that $\phi=J$ on $\aone$;
        \item[(ii)] $\|\phi(x)\phi(y)\|=\|xy\|$, $x,y\in \aone$;
        \item[(iii)] $\|\phi(x)\phi(y)\|_S=\|xy\|_S$, $x,y\in \aone$; 
        \item[(iv)] $\sigma(\phi(x)\phi(y))=\sigma(xy)$, $x,y\in \aone$;
        \item[(v)] For a positive real number $p$, 
        
        $\|(\phi(x)^\frac{p}2\phi(y)^p\phi(x)^\frac{p}2)^\frac1p\|=\|(x^\frac{p}2y^px^\frac{p}2)^\frac1p\|$, $x,y\in \aone$.
       
    \end{itemize}
\end{theorem}
%%%%%%%%%%%%%%%%%%%%%%%%%%%%%%%%%%%%%%%%%%%%
Note that the equivalence between (i) and (ii) is already pointed out just after Theorem 25 (\cite[p.36]{molcon}). For clarity, we include proof. 
%%%%%%%%%%%%%%%%%%%%%%%%%%%%%%%%%%%%%%%%%%%%
\begin{proof}
We prove that (i) implies (ii). Suppose that there exists a Jordan $*$-isomorphism $J$ such that $\phi=J$ on $\aone$. Since $J$ preserves squaring, the involution, the triple product, and the norm, we obtain 
\begin{multline*}
\|\phi(x)\phi(y)\|^2=\|J(x)J(y)\|^2
=\|(J(x)J(y))^*J(x)J(y)\|
\\
=\|J(y)J(x^2)J(y)\|
=\|J(yx^2y)\|
=\|yx^2y\|=\|xy\|^2, 
\end{multline*}
so (ii) holds. 

We prove that (ii) implies (i). 
Suppose that $\|\phi(x)\phi(y)\|=\|xy\|$ holds for every $x,y\in \aone$. Letting $x=y$ we obtain $\|y^2\|=\|\phi(y)^2\|$, hence $\|y\|=\|\phi(y)\|$ since $y$ and $\phi(y)$ are self-adjoint. Then, we have
\[
\|\phi(e)\phi(y)\|=\|ey\|=\|y\|=\|\phi(y)\|=\|e\phi(y)\|
\]
for every $y\in \aone$. As $\phi$ is a surjection we infer by Lemma \ref{1} that $\phi(e)=e$. Put $\psi(x)=\phi(x^\frac12)^2$ for $x\in \aone$.  
 It is apparent that $\psi$ is a surjection. We infer that 
 $\psi(e)=\phi(e^\frac12)^2=\phi(e)^2=e^2=e$. 
Let $x,y$ be arbitrary in $\aone$. Since $yxy=(x^\frac12y)^*x^\frac12y$ and 
\[
(\phi(x^\frac12)\phi(y))^*\phi(x^\frac12)\phi(y)=\phi(y)\phi(x^\frac12)^2\phi(y)=\phi(y)\psi(x)\phi(y)
\]
hold, we have
\begin{equation}\label{yxy}
\|yxy\|=\|x^\frac12y\|^2=\|\phi(x^\frac12)\phi(y)\|^2=\|\phi(y)\psi(x)\phi(y)\|.
\end{equation}
 Then by Lemma \ref{34}
there exists a Jordan $*$-isomorphism $J$ such that $J=\psi$ on $\aone$. We infer by the definition of $\psi$ that $\phi=J$ on $\aone$.

It is apparent that (iv) implies (iii). 

Suppose that (iii) holds. We prove (i). We show that $\phi(e)=e$. Letting $x=y$ in (iii) we have $\|\phi(y)^2\|_S=\|y^2\|_S$. By the spectral mapping theorem, we have $\sigma (z^2)=\{t^2\colon t\in \sigma(z)\}$, so that $\|y^2\|_S=\|y\|_S^2$ and $\|\phi(y)^2\|_S=\|\phi(y)\|_S^2$. Thus we have 
$\|y\|_S=\|\phi(y)\|_S$. 
Thus we have 
\[
\|e\phi(y)\|_S=\|\phi(y)\|_S=\|y\|_S=\|ey\|_S=\|\phi(e)\phi(y)\|_S
\]
for every $y\in \aone$. 
By applying Corollary \ref{3} we have $\phi(e)=e$. 
Since $\sigma(xy)=\sigma(y^\frac12xy^\frac12)$, we have
\begin{equation}\label{ssnashi}
\|xy\|_S=\|y^\frac12xy^\frac12\|_S=\|y^\frac12xy^\frac12\|.
\end{equation}
Similarly, we obtain
\begin{equation}\label{ssnashi2}
\|\phi(x)\phi(y)\|_S=\|\phi(y)^\frac12\phi(x)\phi(y)^\frac12\|.
\end{equation}
Letting $\varphi(y)=\phi(y^2)^\frac12$ for each $y\in \aone$, $\varphi\colon \aone\to \atwo$ is a surjection. By \eqref{ssnashi} and \eqref{ssnashi2} we observe that
\[
\|yxy\|=\|\varphi(y)\phi(x)\varphi(y)\|,\quad x,y\in \aone. 
\]
Then by Lemma \ref{34} there exists a Jordan $*$-isomorphism $J$ such that $J=\phi$ on $\aone$.

A proof that (i) implies (iv) is as follows. As a Jordan $*$-isomorphism $J$ preserves squaring,  the triple product, and the spectrum we have
 \begin{multline*} \sigma(J(x)J(y))=\sigma(J(x)J(y^\frac12)^2)=\sigma(J(y^\frac12)J(x)J(y^\frac12))
 \\
=\sigma(J(y^\frac12xy^\frac12))=\sigma(y^\frac12xy^\frac12)=\sigma(xy).
 \end{multline*}

We prove that (v) implies (i). Let $p=1$ first. 
 As $x^\frac12 yx^\frac12\in \aone$ for every $x,y\in \aone$, we infer
 $\|x^\frac12yx^\frac12\|=\|x^\frac12yx^\frac12\|_S$.  
 We also have 
 $\sigma(x^\frac12yx^\frac12)=\sigma(xy)$
 for every $x,y\in \aone$. Hence, we have
 $\|x^\frac12yx^\frac12\|=\|xy\|_S$. Similarly, 
 we have  $\|\phi(x)^\frac12\phi(y)\phi(x)^\frac12\|=\|\phi(x)\phi(y)\|_S$. Hence, (v) is equivalent to (iii) if $p=1$. 
 We consider a general positive real number $p$. 
 Put $\psi(x)=\phi(x^\frac1p)^p$, $x\in \aone$. By a simple calculation, substituting $x$ by $x^\frac1p$ and $y$ by $y^\frac1p$ in (v), we obtain
 \begin{multline}\label{1p}
 \|(\psi(x)^\frac12\psi(y)\psi(x)^\frac12)^\frac1p\|=
 \|(\phi(x^\frac1p)^\frac{p}2\phi(y^\frac1p)^p\phi(x^\frac1p)^\frac{p}2)^\frac1p\|
 \\
 =
 \|(x^\frac12yx^\frac12)^\frac1p\|, \quad x,y\in \aone.
 \end{multline}
By the spectral mapping theorem we have $\sigma(a^\frac1p)=\sigma(a)^\frac1p$ for every positive invertible element $a$, we have 
\[
\|a^\frac1p\|=\|a^\frac1p\|_S=\|a\|_S^\frac1p=\|a\|^\frac1p
\]
for every positive invertible element $a$. By \eqref{1p} we have
\[
\|\psi(x)^\frac12\psi(y)\psi(x)^\frac12\|=
\|x^\frac12yx^\frac12\|,\quad x,y\in \aone.
\]
It follows from the first part that 
\[
\|\psi(x)\psi(y)\|_S=\|xy\|_S,\quad x,y\in \aone, 
\]
(iii) for $\psi$ is observed. As we have already proven that (iii) and (i) are equivalent, a Jordan $*$-isomorphism $J$ exists, such as $\psi=J$. Then $\phi(x)=\psi(x^p)^\frac1p=J(x^p)^\frac1p=J(x)$, $x\in \aone$; (i) holds for $\phi$. Suppose that $\phi=J$ on $\aone$. Since a Jordan $*$-isomorphism $J$ preserves the power and the triple product, we see that 
\[
(\phi(x)^\frac{p}2\phi(y)^p\phi(x)^\frac{p}2)^\frac1p=
{(J(x^\frac{p}2y^px^\frac{p}2))}^\frac1p,
\]
and $J$ preserves the norm. Thus, we have (v).

%We conclude that (i), (ii), (iii), (vi), and (v) are all equivalent.
\end{proof}
%%%%%%%%%%%%%%%%%%%%%%%%%%%%%%%%%%%%%%%%%%%%%
\begin{theorem}\label{huna}
    Let $\phi_j\colon \aone\to \atwo$ be a bijection for $j=1,2$. 
    Suppose that 
    \[
    \|\phi_1(x)\phi_2(y)\|=\|xy\|,\qquad x,y\in \aone. 
    \]
    Suppose also that $\phi_1(e)=e$. Then, there exists a Jordan $*$-isomorphism $J\colon A_1\to A_2$ such that $\phi_1=\phi_2=J$ on $\aone$.
\end{theorem}
%%%%%%%%%%%%%%%%%%%%%%%%%%%%%%%%%%%%%%%%%%%%%
\begin{proof}
    We have
    \begin{multline}\label{sarudayosaru}
\|\phi_2(y)\phi_1(x^\frac12)^2\phi_2(y)\|=\|\phi_1(x^\frac12)\phi_2(y)\|^2
\\
=\|x^\frac12y\|^2=\|yxy\|, \quad x,y\in \aone.
    \end{multline}
    Put $\psi(x)=\phi_1(x^\frac12)^2$, $x\in \aone$.  Then $\psi\colon \aone\to \atwo$ is a bijection. 
    Apparently,  $\psi(e)=e$. Applying Lemma \ref{34} there is a Jordan $*$-isomorphism $J\colon A_1\to A_2$ such that $\phi_1=J$ on $\aone$. 
        Then we have
    \begin{multline*}
    \|\phi_2(e)J(x)\|=\|(\phi_2(e)J(x))^*\|
    \\
    =
    \|J(x)\phi_2(e)\|=\|xe\|=\|x\|=\|J(x)\|=\|eJ(x)\|
    \end{multline*}
for every $x\in \aone$ we have $e=\phi_2(e)$ as $J$ is surjective. Similarly as $\phi_1$, there is a Jordan $*$-isomorphism $J'\colon \aone\to \atwo$ such that $\phi_2=J'$ on $\aone$. Let $y\in \aone$ be fixed. Then for every $x\in \aone$ we have
\begin{multline*}
\|J(x)J'(y)\|^2=\|xy\|^2=\|xy^2x\|=\|J(xy^2x)\|
\\
=\|J(x)J(y)^2J(x)\|=\|J(x)J(y)\|^2.
\end{multline*}
Then $J'(y)=J(y)$ by Lemma \ref{2}. As $y$ can be arbitrary we conclude
that $J'(y)=J(y)$ for every $y\in \aone$. Thus $\phi_1=\phi_2=J$ on $\aone$. 
\end{proof}
%%%%%%%%%%%%%%%%%%%%%%%%%%%%%%%%%%%%%%%%%%%%%
Note that the assumption $\phi_1(e)=e$  is essential in Theorem \ref{huna} as the following example shows.
%%%%%%%%%%%%%%%%%%%%%%%%%%%%%%%%%%%%%%%%%%%%
\begin{example}
    Suppose that $a\in \aone$ is not a central element and $J\colon A_1\to A_2$ is a Jordan $*$-isomorphism. Let $\phi_1\colon \aone \to \atwo$ be defined as $\phi_1(x)=(aJ(x)^2a)^\frac12$, $x\in \aone$ and $\phi_2(x)=(a^{-1}J(x)^2a^{-1})^\frac12$, $x\in \aone$. 
    We infer that $\phi_1(e)=a$ and  $\phi_1(x)=(\phi_1(e)J(x)^2\phi_1(e))^\frac12$. We also have  $\phi_2(y)=\phi_1(y^{-1})^{-1}$, $y\in \aone$. Hence, we have by Theorem \ref{qim} that 
    \[
\|\phi_1(x)\phi_2(y)\|=\|\phi_1(x)\phi_1(y^{-1})^{-1}\|
%\\
%=\|(aJ(x)^2a)^\frac12(aJ(y^{-1})^2a)^{-\frac12}\|
=\|x(y^{-1})^{-1}\|=\|xy\| 
    \]
    for every $x,y\in \aone$.
    As $a$ is not a central element, $\phi_1$ nor $\phi_2$ are not additive on $\aone$ according to Theorem \ref{qim}. Hence, $\phi_1$ nor $\phi_2$ are not Jordan $*$-isomorphisms. 
\end{example}
%%%%%%%%%%%%%%%%%%%%%%%%%%%%%%%%%%%%%%%%%%%%%
%%%%%%%%%%%%%%%%%%%%%%%%%%%%%%%%%%%%%%%%%%%%%
%%%%%%%%%%%%%%%%%%%%%%%%%%%%%%%%%%%%%%%%%%%%%
\section{Conditions for centrality of elements in the positive definite cones}\label{central}
%%%%%%%%%%%%%%%%%%%%%%%%%%%%%%%%%%%%%%%%%%%%%
%%%%%%%%%%%%%%%%%%%%%%%%%%%%%%%%%%%%%%%%%%%%%
%%%%%%%%%%%%%%%%%%%%%%%%%%%%%%%%%%%%%%%%%%%%%
In this section, we study certain conditions for the centrality of elements in the positive definite cones
The following is a version of Corollary 5 in \cite{kad2}. 
We can prove this by applying Corollary 5 in \cite{kad2}. We can also prove this using Theorem \ref{thecor}. 
Here we provide a proof as a corollary of Proposition 1 in \cite{mol17}. 
%%%%%%%%%%%%%%%%%%%%%%%%%%%%%%%%%%%%%%%%%%%%%%%%%%
\begin{proposition}\label{additive}
Let $T\colon \aone\to \atwo$ be a bijection. 
    Suppose that $T$ is additive; i.e., $T(a+b)=T(a)+T(b)$ for all $a,b\in \aone$. 
    Then there exists a Jordan $*$-isomorphism $J\colon A_1\to A_2$ such that $T(a)=T(e)^{\frac12}J(a)T(e)^{\frac12}$ for every $a\in \aone$.
\end{proposition}
%%%%%%%%%%%%%%%%%%%%%%%%%%%%%%%%%%%%%%%%%%%%%%%%%%%%%%%%%%%%%%%%
\begin{proof}
We prove that $T$ preserves the arithmetic mean; $T\left(\frac{a+b}{2}\right)=\frac{T(a)+T(b)}{2}$ for every $a,b\in \aone$. 
First, for any $a\in \aone$ we have $T(a)=T(\frac{a}{2}+\frac{a}{2})=2T(\frac{a}{2})$. Hence $\frac{T(a)}{2}=T(\frac{a}{2})$ for every $a\in \aone$. It follows that 
\[
T\left(\frac{a+b}{2}\right)=\frac12(T(a+b))=\frac{T(a)+T(b)}{2}
\]
for every $a,b\in \aone$. Thus $T$ preserves the arithmetic mean. Then by \cite[Proposition 1]{mol17} that $T$ is of the desired form. 
\end{proof}
%%%%%%%%%%%%%%%%%%%%%%%%%%%%%%%%%%%%%%%%%%%%%%%%%
Recall that the geometric mean $x\#y$ of $x,y\in A_+^{-1}$ is    
\[
x\#y=x^\frac12(x^{-\frac12}yx^{-\frac12})^\frac12x^\frac12.
\]
Let $a\in A_+^{-1}$ and $\phi(x)=a^2\#x^2$, $x\in A_+^{-1}$. Then $\phi$ is a bijection from $A_+^{-1}$ onto itself. 
Note that if $A$ is commutative, then $a^2\#x^2=ax$. It means that $\phi$ is additive if $A$ is commutative. We have the converse. 
Suppose that $\phi$ is additive. Then $a^{-1}\phi(x) a^{-1}=(a^{-1}x^2a^{-1})^\frac12$  defines an additive bijection from $A_+^{-1}$ onto itself. Then by Proposition \ref{ax2a}, $a^{-1}$ is central, hence so is $a$. 
Similarly, it occurs for $(b\#x)^2$. 
If $A$ is commutative, then $(b\#x)^2=bx$ for every $b,x\in A_{+}^{-1}$. 
 On the other hand, we have the following. 
%%%%%%%%%%%%%%%%%%%%%%%%%%%%%%%%%%%%%%%%%%%%%%
\begin{corollary}
Let $b\in A_+^{-1}$. 
    Suppose that a map $\phi\colon A_+^{-1}\to A_+^{-1}$ which is defined as $\phi(x)=(b\#x)^2$, $x\in A_+^{-1}$ is additive. Then $b$ is a central element in $A$.
\end{corollary}
%%%%%%%%%%%%%%%%%%%%%%%%%%%%%%%%%%%%%%%%%%%%%%%%%%%%%%%%%
\begin{proof}
     We infer that $\phi$ is a bijection by a simple calculation. We have $\phi(e)=(b\# e)^2=b$. Hence, by Proposition \ref{additive} there exists a Jordan $*$-isomorphism $J$ such that $\phi=b^{\frac12}Jb^{\frac12}$. 
     As $b\#x=b^\frac12(b^{-\frac12}xb^{-\frac12})^\frac12 b^\frac12$ we have
     \[
     J(x)=(b^{-\frac12}xb^{-\frac12})^\frac12b(b^{-\frac12}xb^{-\frac12})^\frac12, \quad x\in A_+^{-1}.
     \]
     As a Jordan $*$-isomorphism preserves the triple product, we have
     \begin{equation}\label{bjb}    J(b^\frac12)J(x)J(b^\frac12)=J(b^\frac12xb^\frac12)=x^\frac12 b x^\frac12, \quad x\in A_+^{-1}.
     \end{equation}
     On the other hand, as we have
     \[
     b^\frac12J(b)b^\frac12=\phi(b)=(b\# b)^2=b^2,    
     \]
     we obtain $J(b)=b$, hence $J(b^\frac12)=b^\frac12$. 
     By \eqref{bjb}      
     we obtain 
     \begin{equation}\label{h3}     b^\frac12J(x)b^\frac12=x^\frac12bx^\frac12
     \end{equation}
     for every $x\in A_+^{-1}$. Substituting $x^{-1}$ instead of $x$ we have \begin{equation}\label{j3}     b^\frac12J(x)^{-1}b^\frac12=b^\frac12J(x^{-1})b^\frac12=x^{-\frac12}bx^{-\frac12}. 
     \end{equation}
     Hence
     \begin{equation}\label{j4}
     b^{-\frac12}J(x)b^{-\frac12}=x^\frac12b^{-1}x^\frac12.
     \end{equation}
     By \eqref{h3} and \eqref{j4} we have
     \[
     b^{-\frac12}x^\frac12bx^\frac12b^{-\frac12}=b^\frac12x^\frac12b^{-1}x^\frac12b^\frac12,
     \]
     so
     \[
     (b^\frac12x^\frac12b^{-\frac12})^*(b^\frac12x^\frac12b^{-\frac12})
     =
     (b^\frac12x^\frac12b^{-\frac12})(b^\frac12x^\frac12b^{-\frac12})^*.
     \]
     This means that $b^\frac12x^\frac12b^{-\frac12}$ is normal. 
     Alongside the equation    
     \[
     \sigma(b^\frac12x^\frac12b^{-\frac12})=\sigma(x^{\frac12})\subset (0,\infty),
     \]
     we observe that $b^\frac12x^\frac12b^{-\frac12}$ is self-adjoint, 
     similar to the manner highlighted in the proof of Proposition \ref{ax2a}.
      Hence, we have
     \[
     b^\frac12x^\frac12b^{-\frac12}=(b^\frac12x^\frac12b^{-\frac12})^*=b^{-\frac12}x^\frac12b^\frac12.
     \]
     It follows that $bx^\frac12=x^\frac12b$ holds for every $x$. Then $b$ is a central element in $A$.
\end{proof}
%%%%%%%%%%%%%%%%%%%%%%%%%%%%%%%%%%%%%%%%%%%%%%%%%%%
We present a proof of a localized version of Ogasawara's theorem \cite{oga} concerning commutativity. This result is a specific instance of a theorem of Nagy \cite[Theorem 1]{nagy} and Virosztek \cite[Theorem 1]{virosz} applied to the squaring functions, addressing the local monotonicity of a strictly convex increasing function. 
Our presentation provides a concise and straightforward direct proof, despite the application of two profound results: the Russo and Dye theorem \cite[Corollary 1]{rd} and Kadison's generalized Schwarz lemma \cite[Theorem 1]{kad2}.

%We show a proof of a localized version of a theorem of Ogasawara \cite{oga} about the commutativity. 
%It is just a particular case for the squaring functions of a theorem of Nagy \cite[Theorem 1]{nagy} (cf. \cite[Theorem 1]{virosz}) about the local monotonicity of a strictly convex increasing function. 
%We present a concise and straightforward direct proof.
%establishing that the squaring function is locally monotone at a positive invertible element if and only if it is central.
%%%%%%%%%%%%%%%%%%%%%%%%%%%%%%%%%%%%%%%%%%%%%%%%%%%%%%%%%%
\begin{proposition}\label{lo}
    Let $a\in A_+^{-1}$. Then $a$ is central if and only if $a^2\le x^2$ holds for every $x\in A_+^{-1}$ with $a\le x$.
\end{proposition}
%%%%%%%%%%%%%%%%%%%%%%%%%%%%%%%%%%%%%%%%%%%%%%%%%%%%%%%%%
\begin{proof} If $a$ is central, it is apparent that $a^2\le x^2$ holds for every $x\in A_+^{-1}$ with $a\le x$.

We prove the converse. Suppose that $a\in A_+^{-1}$ which satisfies that $a^2\le x^2$ holds for every $x\in A_+^{-1}$ with $a\le x$.
    We first prove $0\le ab+ba$ for every $b \in A$ with $0\le b$. For any $0<t\le 1$ we have $a\le (1-t)a+t(a+b)$. By the hypothesis, we have $a^2\le ((1-t)a+t(a+b))^2$. By calculation, we obtain that 
    \[
    (2-t)a^2\le (1-t)(2a^2+ab +b a)+t(a+b)^2. 
    \]
    Letting $t\to 0$, we have
    \[
    0\le ab +b a.
    \]
    Define a map $T\colon A\to A$ as $T(y)=\frac12(a^\frac12ya^{-\frac12}+a^{-\frac12}y a^\frac12)$, $y\in A$. 
    Then $T$ is positive since $T(b)=\frac12a^{-\frac12}(ab+b a)a^{-\frac12}\ge 0$ for any $b \ge0$. We also have $T(e)=e$. 
        By a theorem of Russo and Dye \cite[Corollary 1]{rd}, we see that $\|T\|=1$.
        Kadison's generalized Schwarz lemma \cite[Theorem 1]{kad2} asserts that $T(x^2)\ge T(x)^2$ for each self-adjoint element $x$. After calculating $T(x^2)$ and $T(x)^2$, we can observe that
        \[
        a^\frac12x^2a^{-\frac12}+a^{-\frac12}x^2a^\frac12\ge a^\frac12xa^{-1}xa^\frac12+a^{-\frac12}xaxa^{-\frac12}.
        \]
        By multiplying $a^\frac12$ from the left and right-hand sides of each term of both sides of the inequality, we infer that
        \[
        ax^2+x^2a\ge axa^{-1}xa+xax.
        \]
        Then 
        \[
        ax^2-xax\ge axa^{-1}xa-x^2a=(ax-xa)a^{-1}xa.
        \]
        Thus we obtain
        \begin{equation}\label{plus}
        (ax-xa)a^{-1}(ax-xa) \ge 0.
        \end{equation}
        On the other hand, as $(ax-xa)^*=-(ax-xa)$, we obtain 
        \begin{equation}\label{minus}
        0\le ((ax-xa)a^{-\frac12})((ax-xa)a^{-\frac12})^*=-(ax-xa)a^{-1}(ax-xa).
        \end{equation}
        Combining \eqref{plus} and \eqref{minus}, we infer that 
        \[
        \|(ax-xa)a^{-\frac12}\|^2=\|(ax-xa)a^{-1}(ax-xa)\|=0.
        \]
        It follows that $ax=xa$. As $x$ is an arbitrary self-adjoint element, we conclude that $a$ is a central element.
\end{proof}
%%%%%%%%%%%%%%%%%%%%%%%%%%%%%%%%%%%%%%%%%%%%%%%%%%%%%%%%%%%
\begin{proposition}\label{cent}
    Let $a\in A_+^{-1}$. The following are equivalent:
    \begin{itemize}
    \item[(i)]
     $a$ is a central element in $A$,
\item[(ii)]
    $\|axa^{-1}\|=\|x\|$ for every element $x\in A_+^{-1}$,
\item[(iii)]
 $\|a^2x\|=\|axa\|$ for every element $x\in A_+^{-1}$,
\item[(iv)]
     $\|ax\|=\|ax\|_S$ for every element $x\in A_+^{-1}$,
     \item[(v)] 
     $\|a^2x\|=\|a^2x\|_S$ for every element $x\in A_+^{-1}$.
     \end{itemize}
\end{proposition}
%%%%%%%%%%%%%%%%%%%%%%%%%%%%%%%%%%%%%%%%%%%%%%%%%%%
\begin{proof}
    First, we prove that (i) and (iv) are equivalent. 
    Suppose that (i) holds; $a$ is a central element. Then, so is $a^\frac12$ since $a^\frac12$ is approximated by polynomials of $a$. Hence $ax=a^\frac12xa^\frac12$, so $\|ax\|=\|a^\frac12xa^\frac12\|$. As $a^\frac12xa^\frac12$ is positive, so self-adjoint, we have $\|a^\frac12xa^\frac12\|=\|a^\frac12xa^\frac12\|_S$. As $\sigma(a^\frac12xa^\frac12)=\sigma(ax)$, we have $\|a^\frac12xa^\frac12\|_S=\|ax\|_S$. Thus, (iv) holds. 

    Suppose that $a$ is not a central element. We prove (iv) does not hold. As the above, we see that  $a^2$ is not a central element. Applying Proposition \ref{lo}, there exists $x_0\in A_+^{-1}$ such that $a\le x_0^{-1}$ and $a^2\not\le x_0^{-2}$. Put 
    \[
    t_0=\inf\{t\colon a^2\le tx_0^{-2}\},
    \]
    and
    \[
    s_0=\inf\{s \colon a\le sx_0^{-1}\}.
    \]
    Then $t_0>1$ as $a^2\not\le x_0^{-2}$. Since $a\le x_0^{-1}$ we have $s_0\le 1$. We compute
    \[
    t_0=\inf\{t\colon x_0a^2x_0\le te\}
    =
    \sup\{t\colon t\in \sigma (x_0a^2x_0)\}
    =\|x_0a^2x_0\|,
    \]
    where the last equality holds since $x_0a^2x_0$ is positive, hence self-adjoint. Since $x_0a^2x_0=(ax_0)^*(ax_0)$ we have
    \[
    \|x_0a^2x_0\|=\|ax_0\|^2.
    \]
    On the other hand, we compute
    \begin{multline*}
        s_0=\inf\{s\colon a\le sx_0^{-1}\}
        =
        \inf\{s\colon x_0^\frac12ax_0^\frac12\le se\}
        \\
        =
        \sup\{s\colon s\in \sigma (x_0^\frac12ax_0^\frac12)\}
        =
        \sup\{s\colon s\in \sigma(ax_0)\}
        =\|ax_0\|_S.
    \end{multline*}
    Since $t_0 >1\ge s_0^2$, we can conclude that $\|ax_0\|>\|ax_0\|_S$, so (iv) does not hold. 

    In a similar way as above, we see that (i) and (v) are equivalent. 

    By a calculation, we have
    \[
    \|axa\|=\sup\{t\colon t\in \sigma(axa)\}=\sup\{t\colon t\in \sigma(a^2x)\}=\|a^2x\|_S.
    \]
    Hence (iii) and (v) are equivalent. 

        Suppose that (ii) holds. By substituting $axa$ in place of $x$ in equation (ii), we have $\|a^2x\|=\|axa\|$. Thus (iii) holds.

        Suppose that (iii) holds. Substituting $a^{-1}xa^{-1}$ in place of $x$ in equation (iii), we have
        $\|axa^{-1}\|=\|x\|$. Thus (ii) holds.
\end{proof}
%%%%%%%%%%%%%%%%%%%%%%%%%%%%%%%%%%%%%%%%%%%%%%%%%%%%%%%%%%%%%%%%%%%%%%%%%%%%%%%%%%%%%%%%%%
%%%%%%%%%%%%%%%%%%%%%%%%%%%%%%%%%%%%%%%%%%%%%
\section{Maps between positive semidefinite cones}
%%%%%%%%%%%%%%%%%%%%%%%%%%%%%%%%%%%%%%%%%%%%%
%%%%%%%%%%%%%%%%%%%%%%%%%%%%%%%%%%%%%%%%%%%%%
%%%%%%%%%%%%%%%%%%%%%%%%%%%%%%%%%%%%%%%%%%%%%
In this section, we exhibit results on maps between positive semidefinite cones of unital $C^*$-algebras. 
 Theorem \ref{thecor}, "the corollary"  of Moln\'ar, can be slightly generalized for the case of positive semidefinite cones of unital $C^*$-algebras as Theorem \ref{exthe}. It is in the same vein as "the corollary".
Notably,  
the Thompson metric is not well defined on the positive semidefinite cones. However, a representation theorem is possible for positively homogeneous order isomorphisms between positive semidefinite cones. As an application of Theorem \ref{exthe}, we provide a solution (Theorem \ref{semi13}) to the problem posed by  Moln\'ar \cite[p.194]{moljot}. 
%%%%%%%%%%%%%%%%%%%%%%%%%%%%%%%%%%%%%%%%%%%%%%%
%%%%%%%%%%%%%%%%%%%%%%%%%%%%%%%%%%%%%%%%%%%%%%
\begin{theorem}\label{exthe}
    Suppose that $\phi\colon \semione\to \semitwo$ is a positively homogeneous order isomorphism, that is, 
    \begin{itemize}
        \item[(1)] $\phi$ is a bijection,
        \item[(2)] $\phi(tx)=t\phi(x)$, $\quad x\in \semione,\,\,t\ge 0$,
        \item[(3)] $x\le x'$ if and only if $\phi(x)\le \phi(x')$.
    \end{itemize}
    Then there exists a Jordan $*$-isomorphism $J\colon A_1\to A_2$ such that $\phi=\phi(e)^\frac12J\phi(e)^\frac12$ on $\semione$.
\end{theorem}
%%%%%%%%%%%%%%%%%%%%%%%%%%%%%%%%%%%%%%%%%%%%%%%
\begin{proof}
    We first prove that $\phi(\aone)=\atwo$. Suppose that $a\in \aone$. Let $b=\phi^{-1}(e)$, where $e$ is the unit in $A_2$. As $a$ is invertible, $a^{-\frac12}ba^{-\frac12}\ge 0$,
    %$0<\inf\{t\colon t\in \sigma(a)\}$, 
    so there exists a positive real number $t_b>0$ such that $b\le t_ba$. As $\phi$ preserves the order and is positively homogeneous, we infer that 
    \[
    e=\phi(b)\le \phi(t_ba)=t_b\phi(a). 
    \]
    Hence $\phi(a)$ is invertible, so $\phi(a)\in \atwo$. Conversely, we infer in a similar way that  
    if $c\in \atwo$, then $\phi^{-1}(c)\in \aone$.
    %$a\in \aone$ provided that $\phi(a)\in \atwo$ for every $a\in \semione$. 
    Thus we have $\phi(\aone)=\atwo$. 

    According to  "the corollary", there exists a Jordan $*$-isomorphism $J\colon A_1\to A_2$ such that $\phi=\phi(e)^\frac12J\phi(e)^\frac12$ on $\aone$. We prove that the equality also holds on $\semione$. Let $a\in \semione$ be arbitrary. For any $t>0$, $a+te\in \aone$. Then
    \[
    \phi(a)\le \phi(a+te)=\phi(e)^\frac12J(a+te)\phi(e)^\frac12=\phi(e)^\frac12J(a)\phi(e)^\frac12+ t\phi(e)
    \]
    as $\phi$ preserves the order. Letting $t\to 0$ we have
    \[
    \phi(a)\le \phi(e)^\frac12J(a)\phi(e)^\frac12.
    \]
    As $\phi$ is a surjection, there exists $a'\in \semione$ such that $\phi(a')=\phi(e)^\frac12J(a)\phi(e)^\frac12$. 
    Hence we have $\phi(a)\le \phi(a')$, so $a\le a'$ as $\phi$ is an order isomorphism. Next, for any $t>0$ we have
    \[
    \phi(a')=\phi(e)^\frac12J(a)\phi(e)^\frac12\le 
    \phi(e)^\frac12J(a+te)\phi(e)^\frac12=\phi(a+te),
    \]
    as $J$ preserves the order and $a+te\in \aone$. 
    Thus, we have $a'\le a+te$ for every $t>0$. Letting $t\to 0$, we have $a'\le a$. It follows that $a=a'$ and $\phi(a)=\phi(a')=\phi(e)^\frac12J(a)\phi(e)^\frac12$. 
    As $a\in \semione$ is arbitrary, we conclude that $\phi=\phi(e)^\frac12J\phi(e)^\frac12$ on $\semione$.
\end{proof}
%%%%%%%%%%%%%%%%%%%%%%%%%%%%%%
The proposition below is about a characterization of the inequality $a\le b$ for $a,b$ in a positive semidefinite cone of a unital $C^*$-algebra. 
It is a generalization of a part of 
 \cite[Lemma 13]{cmm} which applies to positive definite cones, as well as a part of \cite[Lemma 2.8]{moljot}, which applies to positive semidefinite cones of von Neumann algebras. To prove it, we apply \cite[Lemma 13]{cmm}.
%%%%%%%%%%%%%%%%%%%%%%%%%%%%%%%%%%%%%%%%%%%%%%%%%%%%
\begin{proposition}\label{ineq}
    Let $a,b\in \semino$. Suppose that $\|xax\|\le\|xbx\|$ holds for every $x\in \aplumi$. Then $a\le b$. Suppose conversely that $a\le b$. Then $\|xax\|\le \|xbx\|$ holds for all $x\in \semino$.
    In particular, $a\le b$ if and only if $\|xax\|\le \|xbx\|$ for every $x\in \semino$.
\end{proposition}
%%%%%%%%%%%%%%%%%%%%%%%%%%%%%%%%%%%%%%%%%%%%%
\begin{proof}
    Suppose that $\|xax\|\le\|xbx\|$ for every $x\in \aplumi$. Let $x\in \aplumi$ and $0<t<1$ be arbitrary. We have
\begin{multline}\label{tbt}
    \|x(a+te)x\|\le\|xax\|+\|x(te)x\|\le \|xbx\|+\|x(te)x\|
    \\
    \le \|x(b+te)x\|+2\|x(te)x\|.
\end{multline}
As $0<te\le \sqrt{t}b+te$, we have by Lemma 13 in \cite{cmm} that  
\[
\|x(te)x\|\le \|x(\sqrt{t}b+te)x\|=\sqrt{t}\|x(b+\sqrt{t}e)x\|.
\]
As $0<t<1$ we have $0<b+te\le b+\sqrt{t}e$, so by \cite[Lemma 13]{cmm}, 
\[
\|x(b+te)x\|\le \|x(b+\sqrt{t}e)x\|.
\]
We obtain  that
\[
\|x(b+te)x\|+2\|x(te)x\|\le (1+2\sqrt{t})\|x(b+\sqrt{t}e)x\|.
\]
Then by \eqref{tbt} we have
\[
\|x(a+te)x\|\le (1+2\sqrt{t})\|x(b+\sqrt{t}e)x\|=\|x(1+2\sqrt{t})(b+\sqrt{t}e)x\|.
\]
As $x\in \aplumi$ is arbitrary, by \cite[Lemma 13]{cmm} we have 
\[
a+te\le (1+2\sqrt{t})(b+\sqrt{t}e).
\]
Letting $t\to 0$, we obtain $a\le b$.

Conversely, suppose that $a\le b$. Then for every $x\in \semino$, $0\le xax\le xbx$. Hence 
\[
\|xax\|_S=\sup\{t\colon t\in \sigma(xax)\}\le \sup\{t\colon t\in \sigma(xbx)\}=\|xbx\|_S.
\]
As $xax$ and $xbx$ are self-adjoint, the norm and the spectral seminorm coincide with each other, hence we obtain 
$\|xax\|\le \|xbx\|$ for all $x\in \semino$. 
    
\end{proof}
%%%%%%%%%%%%%%%%%%%%%%%%%%%%%%%%%%%%%%%%%%
%The corollary below is apparent by
% Proposition \ref{ineq}.
%%%%%%%%%%%%%%%%%%%%%%%%%%%%%%%%%%%%%%%%%%%%%
\begin{corollary}\label{nhn}
    Let $a,b\in \semino$. Then, $a=b$ if and only if 
    $\|xax\|=\|xbx\|$ for every $x\in \semino$ with $x\le e$.
\end{corollary}
%%%%%%%%%%%%%%%%%%%%%%%%%%%
\begin{proof}
Suppose that $\|xax\|=\|xbx\|$ for every $x\in \semino$ with $x\le e$. Then for every positive real number $t$, we have $\|(tx)a(tx)\|=\|(tx)b(tx)\|$. Thus we have $\|yay\|=\|yby\|$ for every $y\in \semino$. By Proposition \ref{ineq} we infer that $a=b$. 

The converse statement is trivial. 
\end{proof}
%%%%%%%%%%%%%%%%%%%%
The following proposition is a version of  Lemma \ref{34}.
%%%%%%%%%%%%%%%%%%%%%%%%%%%%%%%%%%%%%%%%%%%%%%%%
\begin{proposition}\label{ppp}
Suppose that $\phi$ and $\psi$ are surjections from $\semione$ onto $\semitwo$ such that
\begin{equation}\label{biyon}
\|yxy\|=\|\phi(y)\psi(x)\phi(y)\|,\quad x,y\in \semione.
\end{equation}
and $\psi(e)=e$. Then, there exists a Jordan $*$-isomorphism $J\colon A_1\to A_2$ such that $\psi=J$ on $\semione$.
\end{proposition}
%%%%%%%%%%%%%%%%%%%%%%%%%%%%%%%%%%%%%%%%%%%%%%%%%%
\begin{proof}
We prove that $\psi$ is an order isomorphism from $\semione$ onto $\semitwo$. Let $x,x'\in \semione$. As 
$\psi$ is a surjection, it is enough to prove that $x\le x'$ if and only if $\psi(x)\le \psi(x')$ for every $x,x'\in \semione$. By \eqref{biyon} and Proposition \ref{ineq}, we have $x\le x'$ if and only if $\|yxy\|\le \|yx'y\|$ for every $y\in \semione$ if and only if 
$\|\phi(y)\psi(x)\phi(y)\|\le \|\phi(y)\psi(x')\phi(y)\|$ for every $y\in \semione$. As $\phi$ is a surjection, it is equivalent to $\psi(x)\le \psi(x')$ by Proposition \ref{ineq}. We conclude that $x\le x'$ if and only if $\psi(x)\le \psi(x')$ for every $x,x' \in \semione$. We infer that $\psi$ is an injection, so it is a bijection. As $x$ and $x'$ are arbitrary, we conclude that $\psi$ is an order isomorphism. We prove that $\psi$ is positively homogeneous. Let $x\in \semione$ and $t\ge 0$ arbitrary. Then we have by \eqref{biyon} that
\begin{multline*}
\|\phi(y)\psi(tx)\phi(y)\|=\|y(tx)y\|
\\
=t\|yxy\|=t\|\phi(y)\psi(x)\phi(y)\|=\|\phi(y)(t\psi(x))\phi(y)\|, \quad y\in \semione.
\end{multline*}
As $\phi$ is a surjection, we have $\psi(tx)=t\psi(x)$ by Corollary \ref{nhn}. Thus, $\psi$ is positively homogeneous. 
Due to Theorem \ref{exthe} and the hypothesis that $\psi(e)=e$, we conclude that there exists a Jordan $*$-isomorphism $J\colon A_1\to A_2$ such that $\psi=J$.
\end{proof}
%%%%%%%%%%%%%%%%%%%%%%%%%%%%%%%%%
The following theorem provides a solution to the problem posed by Moln\'ar \cite[p.194]{moljot}.
%%%%%%%%%%%%%%%%%%%%%%%%%%%%%%%%%%
\begin{theorem}\label{semi13}
    Suppose that $\phi\colon \semione\to \semitwo$ is a surjection. The following are equivalent:
    \begin{itemize}
        \item[(i)] There exists a Jordan $*$-isomorphism $J\colon A_1\to A_2$ such that $\phi=J$ on $\semione$;
        \item[(ii)] $\|\phi(x)\phi(y)\|=\|xy\|$, $\quad x,y\in \semione$;
        \item[(iii)] $\|\phi(x)\phi(y)\|_S=\|xy\|_S$, $\quad x,y\in \semione$,
        \item[(iv)] $\sigma(\phi(x)\phi(y))=\sigma(xy)$, $\quad x,y\in \semione$,
        \item[(v)] For a positive real number $p$, 
        \[
        \|(\phi(x)^\frac{p}{2}\phi(y)^p\phi(x)^{\frac{p}{2}})^\frac1p\|=\|(x^{\frac{p}{2}}y^px^{\frac{p}{2}})^\frac1p\|, \quad x,y\in \semione.
        \]
    \end{itemize}
\end{theorem}
%%%%%%%%%%%%%%%%%%%%%%%%%%%%%%%%%%%%%%%%%%%%%%%
\begin{proof}
Suppose that (ii) holds. We prove (i). As $(ab)^*ab=ba^2b$ for every $a,b\in \semione$, we infer by the equation (ii) that 
\begin{equation}\label{ippon}
\|yxy\|=\|x^\frac12y\|^2=\|\phi(x^\frac12)\phi(y)\|^2=\|\phi(y)\phi(x^\frac12)^2\phi(y)\|, \quad x,y\in \semione
\end{equation}
Defining $\psi\colon \semione\to \semitwo$ by $\psi(x)=\phi(x^\frac12)^2$, $x\in \semione$, $\psi$ is a surjection, and  by \eqref{ippon} we obtain
\begin{equation}\label{nihon}
\|yxy\|=\|\phi(y)\psi(x)\phi(y)\|,\quad x,y\in \semione.
\end{equation}
We prove $\psi(e)=e$. Letting $x=y$ in the equation of (ii), we have $\|y^2\|=\|\phi(y)^2\|$. 
%Hence $\|y\|=\|\phi(y)\|$ for every $y$ since $y$ and $\phi(y)$ are self-adjoint. 
Thus we have
\begin{multline*}
\|\phi(y)\phi(e)^2\phi(y)\|=\|\phi(e)\phi(y)\|^2
\\
=\|ey\|^2=\|y^2\|=\|\phi(y)^2\|=\|\phi(y)e\phi(y)\|.
\end{multline*}
By Corollary \ref{nhn}, we have $\phi(e)^2=e$, so $\phi(e)=e$, hence $\psi(e)=e$. 
Applying Proposition \ref{ppp}, there exists a Jordan $*$-isomorphism $J\colon A_1\to A_2$ such that $\psi=J$ on $\semione$. For any $x\in \semione$, $\phi(x)=\psi(x^2)^\frac12=J(x^2)^\frac12=J(x)$. Thus $\phi=J$ on $\semione$.

We prove (v) implies (i). Suppose first we consider $p=1$. In this case, the equation of (iii) is 
\begin{equation}\label{12112}
\|\phi(x)^\frac12\phi(y)\phi(x)^\frac12\|=\|x^\frac12yx^\frac12\|, \quad x,y\in \semione.
\end{equation}
We prove $\phi(e)=e$. Inserting $x=y$ in \eqref{12112}, we obtain $\|\phi(x)^2\|=\|x^2\|$. Hence $\|\phi(x)\|=\|x\|$ for every $x\in \semione$. 
Then by \eqref{12112} we have
\begin{multline*}
    \|\phi(x)^\frac12\phi(e)\phi(x)^\frac12\|=\|x^\frac12ex^\frac12\|
    \\
    =\|x\|=\|\phi(x)\|=\|\phi(x)^\frac12e\phi(x)^\frac12\|,\quad x\in \semione.
\end{multline*}
As $\phi$ is a surjection, we see that $\{\phi(x)^\frac12\colon x\in \semione\}=\semitwo$. Then Corollary \ref{nhn} ensures that $\phi(e)=e$. 
Define $\varphi\colon \semione\to \semitwo$ by
$\varphi(x)=\phi(x^2)^\frac12$, $x\in \semione$. Then by \eqref{12112} we obtain
\[
\|\varphi(x)\phi(y)\varphi(x)\|=\|xyx\|,\quad x,y\in \semione.
\]
Then by Proposition \ref{ppp} we observe that (i) holds.

We consider a general positive $p$. 
Define $\phi'\colon \semione\to \semitwo$ by $\phi'(x)=\phi(x^\frac1p)^p$, $x\in \semione$. By a simple calculation, we infer that $\phi'$ is a surjection and 
\[
\|x^\frac12yx^\frac12\|=\|\phi'(x)^\frac12\phi'(y)\phi'(x)^\frac12\|, \quad x,y\in \semione.
\]
By the first part, a Jordan $*$-isomorphism $J\colon A_1\to A_2$ exists such that $\phi'=J$. Thus, we also see that $\phi=J$ on $\semione$.

It is apparent that (v) implies (iv). 

Suppose that (iii) holds. Let $x,y\in \semione$ arbitrary. By the so-called Jacobson's lemma, we have $\sigma(xy)\setminus\{0\}=\sigma(x^\frac12yx^\frac12)\setminus\{0\}$. Hence we 
have $\|xy\|_S=\|x^\frac12yx^\frac12\|_S$. As $x^\frac12yx^\frac12$ is self-adjoint, we have 
$\|x^\frac12yx^\frac12\|_S=\|x^\frac12yx^\frac12\|$. 
Therefore we have $\|xy\|_S=\|x^\frac12yx^\frac12\|$. In the same way we have 
$\|\phi(x)\phi(y)\|_S=\|\phi(x)^\frac12\phi(y)\phi(x)^\frac12\|$. Thus we have 
\[
\|\phi(x)^\frac12\phi(y)\phi(x)^\frac12\|=\|x^\frac12yx^\frac12\|,\quad x,y\in \semione.
\]
In other words, we have (v) for $p=1$ holds for $\phi$. In this case, we have already proven that (i) holds.  

Suppose that (i); $\phi=J$ on $\semione$ for a Jordan $*$-isomorphism $J$. We prove (iv). Let $x,y\in \semione$ 
arbitrary. By the Jacobson lemma, $\sigma(xy)\setminus\{0\}=\sigma(x^\frac12yx^\frac12)\setminus\{0\}$. By a simple calculation we have $x^\frac12yx^\frac12$ is invertible if and only if $xy$ is invertible. (Suppose that $x^\frac12yx^\frac12$ is invertible; 
$(x^\frac12yx^\frac12)a=a(x^\frac12yx^\frac12)=e$ for some 
$a\in A_1$. It means that $x^\frac12$ is also invertible. 
Hence $yx=x^{-\frac12}(x^\frac12yx^\frac12)x^\frac12$ is 
invertible. Thus $xy=(yx)^*$ is invertible. Conversely, suppose that $xy$ is invertible. As $(xy)^*=yx$ is also invertible, we have $xyyx=xy^2x$ is invertible. Hence, there exists $b\in A_1$ such that $bxy^2x=xy^2xb=e$. Hence, $x$ is also invertible. Thus, the spectrum of $x$ consists of positive real numbers. By the spectral mapping theorem, so is the spectrum of $x^\frac12$. Thus, $x^\frac12$ is invertible. Hence $x^\frac12(yx)x^{-\frac12}=x^\frac12yx^\frac12$ is invertible.)
It follows that $\sigma(x^\frac12yx^\frac12)=\sigma(xy)$.  Similarly, we have $\sigma(J(x)^\frac12J(y)J(x)^\frac12)=\sigma(J(x)J(y))$. As a Jordan $*$-isomorphism preserves the spectrum, the triple product, and the square root, we see that
\[
\sigma(x^\frac12yx^\frac12)=\sigma(J(x^\frac12yx^\frac12))=\sigma(J(x)^\frac12J(y)J(x)^\frac12).
\]
We conclude that $\sigma(J(x)J(y))=\sigma(xy)$.

We prove that (i) implies (ii). Suppose that there exists a Jordan $*$-isomorphism $J$ such that $\phi=J$ on $\semione$. Since $J$ preserves squaring, the involution, the triple product, and the norm, we obtain for arbitrary $x,y\in \semione$ that
\begin{multline*}
\|\phi(x)\phi(y)\|^2=\|J(x)J(y)\|^2
=\|(J(x)J(y))^*J(x)J(y)\|
\\
=\|J(y)J(x^2)J(y)\|
=\|J(yx^2y)\|
=\|yx^2y\|=\|xy\|^2, 
\end{multline*}
so (ii) holds. 

As a Jordan $*$-isomorphism preserves the triple product and the power, it is easy to prove that (i) implies (v).
\end{proof}
%%%%%%%%%%%%%%%%%%%%%%%%%%%%%%%%%%%%%%%%%%%%%%%
The set $E(A)=\{a\in \semino\colon 0\le a\le e\}$ is known as the effect algebra. 
This structure holds significant importance in the quantum theory of measurements, as it establishes the foundation for unsharp measurements within quantum mechanics.
The following result generalizes Corollary 2.9 in \cite{moljot}, which was stated for von Neumann algebras. Let $p>0$. We define $x\diamond_p y=(x^{\frac{p}{2}}y^px^{\frac{p}{2}})^\frac{1}{p}$ for $x,y\in E(A)$.  The involved operation for $p=1$ is the usual sequential product. 
%%%%%%%%%%%%%%%%%%%%%%%%%%%%%%%%%%%%%%%%%%%%%%
\begin{corollary}
    Let $\phi\colon \effone\to \efftwo$ be a surjection and $p>0$. It satisfies 
    \[
    \|\phi(x)\diamond_p \phi(y)\|=
    \|x\diamond_p y\|,
    %\|(\phi(x)^{\frac{p}{2}}\phi(y)^p\phi(x)^{\frac{p}{2}})^\frac1p\|=\|(x^{\frac{p}{2}}y^px^{\frac{p}{2}})^\frac1p\|,
    \quad x,y\in \semione
    \]
    if and only if there exists a Jordan $*$-isomorphism $J\colon A_1\to A_2$ such that $\phi=J$ on $\effone$.
\end{corollary}
%%%%%%%%%%%%%%%%%%%%%%%%%%%%%%%%%%
\begin{proof}
We first prove that $\phi(tx)=t\phi(x)$ for every $x\in \effone$ and $0<t\le 1$. 
Suppose that $x,y\in \effone$ and $0<t\le 1$. Then we have $ty\in \effone$ and 
\begin{multline*}
\|(\phi(x)^\frac{p}{2}\phi(ty)^p\phi(x)^\frac{p}{2})^\frac{1}{p}\|=\|(x^\frac{p}{2}(ty)^px^\frac{p}{2})^\frac{1}{p}\|=t\|(x^\frac{p}{2}y^px^\frac{p}{2})^\frac{1}{p}\|
\\
=t\|(\phi(x)^\frac{p}{2}\phi(y)^p\phi(x)^\frac{p}{2})^\frac{1}{p}\|=
\|(\phi(x)^\frac{p}{2}(t\phi(y))^p\phi(x)^\frac{p}{2})^\frac{1}{p}\|.
\end{multline*}
As $\{\phi(x)^\frac{p}{2}\colon x\in \effone\}=\efftwo$, we have $\phi(ty)=t\phi(y)$ by Corollary \ref{nhn}.

    Define $\tilde{\phi}\colon \semione\to \semitwo$ as 
    \begin{equation}
        \tilde{\phi}(a)=
        \begin{cases}
            \|a\|\phi\left(\frac{a}{\|a\|}\right), & a\ne 0
            \\
            0, & a=0.
        \end{cases}
    \end{equation}
    It is easy to see that $\tilde{\phi}$ is a surjection. By the first part, we infer that $\phi=\tilde{\phi}$ on $\effone$.  We also see by a simple calculation that
    \[
    \|(\tilde{\phi}(x)^\frac{p}{2}\tilde{\phi}(y)^p\tilde{\phi}(x)^\frac{p}{2})^\frac{1}{p}\|
    =
    \|(x^\frac{p}{2}y^px^\frac{p}{2})^\frac{1}{p}\|,\quad x,y\in \semione.
    \]
    It follows by Theorem \ref{semi13} that there exists a Jordan $*$-isomorphism $J\colon A_1\to A_2$ such that 
    $\tilde{\phi}=J$ on $\semione$, in particular, on $\effone$. As $\phi=\tilde{\phi}$ on $\effone$, we have the conclusion.
\end{proof}
%%%%%%%%%%%%%%%%%%%%%%%%%%%%%%%%%%%%%%%%%%%%%
%%%%%%%%%%%%%%%%%%%%%%%%%%%%%%%%%%%%%%%%%%%%%
%%%%%%%%%%%%%%%%%%%%%%%%%%%%%%%%%%%%%%%%%%%%%

%%%%%%%%%%%%%%%%%%%%%%%%%%%%%%%%%%%%%%%%%%%%%%%%%%%%%%%%%%%%%%%%%%%%%%%%%%%%%%%%%%%%%%%%%%%%%%%%%%%%%%%%%%%%%%%%%%%%%%%%%%%%%%%%%%%%%%%%%%%%%%%%%%%%%%%%%%%%%%%%%%%%%%%%%%%%%%%%%%%%%%%%%%%%%%%%%%%%%%%
\subsection*{Acknowledgments}
The authors would like to thank Lajos Moln\'ar for his fascinating lectures at Ritsumeikan University in 2023. We enjoyed them very much. They also appreciate his comments on the first draft which improved the readability of the paper.
The first author was supported by 
JSPS KAKENHI Grant Numbers JP19K03536. 
The second author was supported by JSPS KAKENHI Grant Numbers JP21K13804. 
%%%%%%%%%%%%%%%%%%%%%%%%%%%%%%%%%%%%%%%%%%%%%%%%%%%%%%%%%%%%%%%%%%%%%%%%%%%%%%%%%%%%%%%%%%%%%%%%%%%%%%%%%%%%%%%%%%%%%%%%%%%%%%%%%%%%%%%%%%%%%%%%%%%%%%%%%%%%%%%%%%%%%%%%%%%%%%%%%%%%%%%%%%%%%%

\end{document}